\documentclass[12pt]{article}

\usepackage{setspace}
\usepackage{amsmath}
\usepackage{amsfonts}
\usepackage{amsthm}

\newcommand{\bp}{\mathcal{P}}
\newcommand{\Disc}{\mathbb{D}}
\newcommand{\conj}[1]{\overline{#1}}
\DeclareMathOperator{\Rp}{Re}

\DeclareMathOperator*{\esssup}{ess\,sup}

\newtheorem{theorem}{Theorem}
\newtheorem{definition}{Definition}
\newtheorem{lemma}[theorem]{Lemma}
\newtheorem{corollary}[theorem]{Corollary}

\begin{document}

\bibliographystyle{amsplain}

\title{Bounds on Integral Means of Bergman Projections and their Derivatives}
\author{Timothy Ferguson\thanks{Department of Mathematics, University of 
Alabama, Box 870350, Tuscaloosa, AL 35487, \hbox{tjferguson1@ua.edu}} \thanks{Thanks to Joseph Cima and an anonymous referee for their 
helpful comments.}\thanks{Partial Support for this work has been provided by the University of Alabama RGC-2015-22 grant}}

\maketitle

\begin{abstract}
We bound integral means of the Bergman projection of a function 
in terms of integral means of the original function. As an application 
of these results, we bound certain weighted Bergman space norms of 
derivatives of Bergman projections in terms of weighted 
$L^p$ norms of certain derivatives 
of the original function in the $\theta$ direction.  
These results easily imply the well known result that 
the Bergman projection is bounded from the Sobolev space 
$W^{k,p}$ into itself for $1 < p < \infty$.
We also apply our results to derive certain regularity results 
involving extremal problems in Bergman spaces. 
Lastly, we construct a 
function that approaches $0$ uniformly at the boundary of the unit disc
but whose Bergman projection is 
not in $H^2$. 
\end{abstract}

For $0 < p < \infty$, the Bergman space $A^p = A^p(\Disc)$ is the space 
of all analytic functions in the unit disc $\Disc$ such that 
\[
\|f\|_{A^p} = \left[ \int_{\Disc} |f(z)|^p d\sigma(z) \right]^{1/p} < \infty.
\]
Here, $\sigma$ is normalized Lebesgue area measure, so that 
$\sigma(\Disc) = 1$. The Bergman spaces are closed subspaces of 
$L^p(\Disc)$ (see \cite{D_Ap} or \cite{Zhu_Ap}). 

For a function in $L^p$ for $0 < p < \infty$, we 
define its $p^{\textrm{th}}$ integral mean at radius $r$ by 
\[
M_p(r,f) = \left[ \frac{1}{2\pi} \int_0^{2\pi} |f(re^{i\theta})|^p 
                       \, d\theta \right]^{1/p}.
\]
If $p = \infty$, we can define 
$M_p(r,f) = \esssup_{0 \le \theta < 2\pi} |f(re^{i\theta})|$. 
It is well known that if $f$ is analytic, then the integral means are 
nondecreasing functions of $r$ (see \cite{D_Hp}).  

For $0 < p \le \infty$, the Hardy space consists of all analytic 
functions in $\Disc$ 
for which $\|f\|_{H^p} = \sup_{0\le r < 1} M_p(r,f) < \infty$. 
It is easy to see that $H^p \subset A^p$ for $0 < p < \infty$. 
In fact, $H^{2p} \subset A^p$, and the $H^{2p}$ norm is always 
greater than or equal to the $A^p$ norm (see \cite{Dragan_Isoperimetric}), 
a fact which is related to the isoperimetric inequality.

If $f \in L^1( \Disc)$, we define its Bergman projection 
to be
\[
\bp f(z) = \frac{1}{\pi} 
    \int_{\Disc} \frac{f(w)}{(1-\conj{w}z)^{2}} dA(w)
\]
for $z \in \mathbb{D}$. 
The function $\bp(f)$ is an analytic function in the unit disc.  
When restricted 
to $L^2(\Disc)$, the Bergman projection is the orthogonal projection 
onto $A^2( \Disc)$.  It is well known that the Bergman projection 
is bounded from $L^p$ to $A^p$ 
for $1 < p < \infty$ (see \cite{D_Ap} or \cite{Zhu_Ap}). 

The main result of this article bounds integral means of derivatives 
of the Bergman projection of a function in terms of integral means of 
angular derivatives of the original function.  
These bounds are then used to bound certain weighted 
Bergman space norms of derivatives of 
the Bergman projection of a function in terms of certain 
weighted $L^p$ norms of derivatives of the original function in the 
$\theta$ direction.  
(See the articles 
\cite{MR2965249} and \cite{MR3130312} for similar results in the 
context of several complex variables.)
Our results easily imply the well known result that 
$\bp$ is bounded from the Sobolev space 
$W^{k,p}$ into itself for $1 < p < \infty$, where $k$ is a nonnegative 
integer.  Lastly, we give a result in the opposite direction from our 
main result: there exists a function 
$f$ such that $f(re^{i\theta}) \rightarrow 0$ uniformly as 
$r \rightarrow 1^{-}$, but for which the integral means
$M_2(r,\bp f)$ are not bounded in $r$. 

We remark that even though our methods are focused on estimating integral 
means of Bergman projections, they allow us to obtain the 
bound $2\pi/\sin(\pi/p)$ for 
the norm of the Bergman projection from $L^p$ to $A^p$.  
It is known that the norm is at least $1/(2 \sin(\pi/p))$ 
and at most $\pi/\sin(\pi/p)$, so that our bound differs from the 
norm by a factor that is between $1/(4\pi)$ and $1/2$ for each $p$ (see
\cite{Dostanic_BP_Norm}).  
In fact, our estimate holds for 
the operator with kernel $1/|1-\overline{\zeta} z|^2$, and it is know that 
the norm of this operator is exactly $\pi/\sin(\pi/p)$ 
(see \cite{Dostanic_Berezin_Norm}).

It may seem unusual to investigate integral means of Bergman projections, 
since integral means are 
related to Hardy spaces and the Bergman projection is related to 
Bergman spaces, so we give some motivation. 
In \cite{Ryabykh}, Ryabykh found a relation between Hardy spaces and 
extremal problems in Bergman spaces.  More specifically, he proved 
the following theorem:
Let $1 < p < \infty$ and let $1/p + 1/q = 1.$  Suppose that 
$\phi \in (A^p)^*$ and that $\phi(f) = \int_{\Disc} f \conj{k} \, d\sigma$ 
for some $k \in H^q,$ where $k \ne 0.$  
(The function $k$ is called the integral kernel of $\phi$.)  
Then the solution to the extremal problem 
of finding the function $F \in A^p$ of unit norm that maximizes 
$\Rp \phi(F)$ 
belongs to 
$H^p$.  (It is known that such an $F$ is unique.)  
Also, $F$ satisfies the bound
\begin{equation*}
\|F\|_{H^p} \le \Bigg\{ \bigg[ \max(p-1,1)
\bigg]\frac{C_p\|k\|_{H^q}}{\| k \|_{A^q}}\Bigg\}^{q/p},
\end{equation*}
where $C_p$ is a constant depending on $p$, which may be taken to be the 
norm of the Bergman projection on $A^p$ (see
\cite{tjf1}, Theorem 4.2).

Other relations between the regularity of $k$ and the regularity of 
$F$ are given in \cite{tjf:lipschitzbergman}.  
To state one, let $1 < p \leq \infty$, and 
say that $F \in \Lambda_{\beta}$ for $0 < \beta \leq 1$ if 
$\|F(e^{it} \cdot) - F(\cdot)\|_{H^p} \leq C|t|^\beta$ for some constant 
$C$, and 
say that $F \in \Lambda^*_{\beta,p}$ for $0 < \beta \leq 2$ if 
$\|F(e^{it} \cdot) + F(e^{-it} \cdot) - 2F(\cdot)\|_{H^p} \leq C|t|^\beta$ 
for some constant $C$.  It is known that 
$\Lambda^*_{\beta,p} = \Lambda_{\beta,p}$ for $0 < \beta < 1$, and that 
$f \in \Lambda_{\beta,p}$ if and only if $M_p(r,f') = O((1-r)^{\beta-1})$, and 
that $f \in \Lambda^*_{\beta,p}$ if and only if 
$M_p(r,f'') = O((1-r)^{\beta-2})$.  
Theorem 4.1 and Corollary 3.5 
of \cite{tjf:lipschitzbergman} imply that 
if 
$1/q < \beta < 2$ and $\beta/\nu > 1/p$ and 
$k \in \Lambda^*_{\beta-(1/q),q}$ then $F \in \Lambda^*_{(\beta/\nu)-(1/p),p}$, 
where $\nu = 2$ if $1 < p < 2$ and $\nu = p$ if $2 < p < \infty$.

The study of bounds for Bergman projections is related to results that 
are in some sense converse to the above results.  This is because 
$k$ is a constant multiple of $\bp(|F|^{p-2}F)$.  
In this paper, we prove results of such type. For example, 
if
$1 \le p - 1 \le p_1 < \infty$ and $0 < \alpha < 1$ and  
the extremal function $F$ is in $H^{p_1}$ and is in 
the space $\Lambda_{\alpha}^{p_1}$, then 
the integral kernel
$k \in H^{q_1}$ and 
has boundary values 
in $\Lambda_{\alpha}^{q_1}$, where $q_1 = p_1 / (p-1)$. 

In \cite{tjf2}, it is  proved 
that the converse to Ryabykh's theorem holds when $p$ is an 
even integer.  In fact, Theorem 4.3 in the above reference 
says that the following holds:
Suppose $p$ is an even integer and let $q$ be its conjugate 
exponent. 
Let $F \in A^p$ with $\|F\|_{A^p}=1$, 
and let $k$ be an integral 
kernel such that $F$ is the extremal function for 
the functional corresponding to $k$. 
(It is known that $k$ is unique up to a positive scalar multiple, 
see \cite{tjf1}). 
If $F \in H^{p_1}$ for some $p_1$ with $p-1 < p_1 < \infty$, then   
$k \in H^{q_1}$ for $q_1 = p_1/(p-1)$, and
\begin{equation*}
\frac{\|k\|_{H^{q_1}}}{\|k\|_{A^q}} \le
C \|F\|_{H^{p_1}}^{p-1},
\end{equation*} 
where $C$ is a constant depending only on $p$ and $p_1$.
(The statement in the reference is only for $p_1 \ge p$, but the 
proof works for all 
$p_1 > p-1$.) 
Since by Theorem 2.2 in \cite{tjf3} the function
$k$ is a constant multiple of 
$\bp(|F|^{p-2} F)$, 
this theorem implies the following result:
Suppose that $1 < q_1 < \infty$ and that $p$ is an even integer.
If $g$ has the form 
$g = |f|^{p-2} f$ for some analytic function $f$, and 
if $g$ has bounded
$M_{q_1}$ integral means, then 
$\bp(g) \in H^{q_1}$, where $\bp$ is the Bergman projection. 

In this paper, we provide a counterexample to a possible generalization 
of this result.  We find a function $g$ with bounded $M_2$ integral means, 
but such that $\bp(g) \not \in H^2$. In fact, we can even take 
$g$ so that $M_\infty(r,g)$ is bounded and $M_\infty(r,g) \rightarrow 0$ 
as $r \rightarrow 1^{-}$, and we can even assume that 
$g \in C^\infty(\Disc)$.  This shows that functions of the form 
$|f|^{p-2} f$, where $f$ is analytic and $p$ is an even integer, 
are in some sense better behaved than general functions under the 
Bergman projection.  It is unknown whether the same type of result 
holds for $p$ not an even integer.

\section{Hypergeometric Functions and Two Lemmas}
We first discuss hypergeometric functions, since we will use them 
in some of our proofs. 
The hypergeometric function ${}_2 F_1(a,b;c;z)$ is defined
for $|z| < 1$ and for $c$ not a non-positive integer by
\[
{}_2 F_1(a,b;c;z) = \sum_{n=0}^\infty \frac{(a)_n (b)_n}{(c)_n n!} z^n 
\]
(see \cite{NIST:DLMF}, eq.~15.2.1), where 
$(a)_n = a(a+1) \cdots (a+n-1) = \Gamma(a+n)/\Gamma(a)$.  
Note that $(a)_0 = 1$.  
(Note that if $c$ is a non-positive integer then not all terms in the 
sum are defined, which is why such values of $c$ are excluded.)
A hypergeometric function may be analytically 
continued to a single valued analytic function on $\mathbb{C}$ minus 
the part of the real axis from $1$ to $\infty$.  This analytic continuation 
is called the principal branch of the hypergeometric function. 

For $\Rp c > \Rp b > 0$, we have that  
\begin{equation}\label{eq_euler_hypergeo_int}
{}_2F_1(a,b;c;z) = \frac{\Gamma(c)}{\Gamma(b)\Gamma(c-b)} 
\int_0^1 x^{b-1} (1-x)^{c-b-1} (1-zx)^{-a} dx
\end{equation}
(see \cite{NIST:DLMF}, eq.~15.6.1).
Also, for $|\arg(1-z)|<\pi$ we have 
\begin{equation}\label{eq_euler_hypergeo_transf}
{}_2F_1(a,b;c;z) = (1-z)^{c-a-b} {}_2F_1(c-a, c-b;c;z)
\end{equation}
(see \cite{NIST:DLMF}, eq.~15.8.1).
Kummer's quadratic transformation states that for $|z|<1$ we have
\begin{equation}\label{eq_gauss_hypergeo_transf}
{}_2F_1(a,b;2b;4z/(1+z)^2) = (1+z)^{2a} {}_2F_1(a,a+\frac{1}{2}-b;
b + \frac{1}{2}; z^2)
\end{equation}
(see \cite{NIST:DLMF}, eq.~15.8.21). 
If $\Rp c > \Rp (a+b)$ then the power series defining 
$_{2}F_{1}$ converges absolutely on the circle $|z|=1$ 
(see \cite{NIST:DLMF}, 15.2.(i)), and thus uniformly on $|z| \le 1$. 
In addition, for $\Rp c > \Rp (a+b)$ the value of the hypergeometric 
function at $1$ (that is, the sum of the hypergeometric series for 
$z=1$) is given by
\begin{equation}\label{eq_hypergeo_value_1}
{}_2F_1(a,b;c;1) =
\frac{\Gamma(c)\Gamma(c-a-b)}{\Gamma(c-a)\Gamma(c-b)}
\end{equation}
(see \cite{NIST:DLMF}, eq.~15.4.20).

The following lemma is well known, although we do not know if anyone has 
found the sharp constant before (see e.g.~\cite{Zhu_Ap}, Theorem 1.7). 
\begin{lemma} \label{lemma:1minusrep}
Let $p>1$ and $0<r<1$. Then 
\[
\begin{split}
\frac{1}{2\pi} \int_0^{2\pi} \frac{1}{|1-re^{i\theta}|^p} d\theta &= 
(1-r^2)^{1-p} {}_2F_1\left(1-\frac{p}{2}, 1-\frac{p}{2};1;r^2\right)
\\ &\le 
\frac{\Gamma(p-1)}{\Gamma(p/2)^2} (1-r^2)^{1-p}.
\end{split}
\]
Furthermore, the bound is sharp, in the sense that the integral in question, 
divided by $(1-r^2)^{1-p}$, is always less than or equal to
$\tfrac{\Gamma(p-1)}{\Gamma(p/2)^2}$, but the quotient approaches 
$\tfrac{\Gamma(p-1)}{\Gamma(p/2)^2}$ as $r \rightarrow 1$. 
\end{lemma}
In the case $p=2$, the equality says that 
\[
\frac{1}{2\pi} \int_0^{2\pi} \frac{1}{|1-re^{i\theta}|^2} d\theta = 
(1-r^2)^{-1}.
\]
\begin{proof}
The integral in question is equal to
\[
\frac{1}{\pi} \int_0^\pi (1-2r\cos\theta + r^2)^{-p/2}\, d\theta.
\]
Making the substitution $x = (\cos \theta + 1)/2$, we see that the
integral is equal to 
\[
\begin{split}
&\ \frac{1}{\pi} (1+r)^{-p} \int_0^1 \left(1 -
  \frac{4r}{(1+r)^2}x\right)^{-p/2} x^{-1/2} (1-x)^{-1/2} dx 
\\ &= 
(1+r)^{-p} {}_2F_1(p/2, 1/2;1; 4r/(1+r)^2)
\end{split}
\]
by equation \eqref{eq_euler_hypergeo_int}.  Now using equation 
\eqref{eq_gauss_hypergeo_transf}, we see this is equal to 
\[
{}_2F_1(p/2, p/2;1;r^2).
\]
Equation \eqref{eq_euler_hypergeo_transf} shows this is equal to 
\[
(1-r^2)^{1-p} {}_2F_1(1-(p/2), 1-(p/2); 1; r^2).
\]
The bound now follows from equation \eqref{eq_hypergeo_value_1}, 
since the series representation shows that 
$_2F_1(1-(p/2), 1-(p/2); 1; r^2)$ increases from $r=0$ to $r=1$. 
The remark about $p=2$ is true because 
$_2F_1(1, 1;1;x) =(1-x)^{-1}.$
\end{proof}

The following lemma is likely well known, at least without 
the sharp constant,
although we do not know of a specific place where it appears in the 
literature.
\begin{lemma}\label{yam1}
Suppose that $s < 1$ and $m+s > 1$ and that $k > -1$. 
Let $0 \le x < 1$. Then 
\[
\int_0^1 \frac{(1-y)^{-s}}{(1-xy)^m} y^k \, dy \le C_1(s,m,k) (1-x)^{1-s-m}
\]
where $C_1(s,m,k) < \infty$ is defined by 
\[
C_1(s,m,k) = \frac{\Gamma(k+1)\Gamma(1-s)}{\Gamma(2+k-s)} 
\max_{0 \le x \le 1} {}_2F_1( 2+k-s-m, 1-s; 2+k-s; x).
\]
Furthermore, the bound is sharp
in the sense that the integral in question, 
divided by $(1-x)^{1-s-m}$, is always less than or equal to
$C_1(s,m,k)$, and furthermore $C_1(s,m,k)$ is the smallest 
constant with this property.
\end{lemma}
\begin{proof}
By \eqref{eq_euler_hypergeo_int}, we have 
\[
\int_0^1 \frac{(1-y)^{-s}}{(1-xy)^m} y^k \, dy = 
\frac{\Gamma(k+1)\Gamma(1-s)}{\Gamma(2+k-s)} {}_2F_1(m,k+1;2+k-s;x)
\]
Now \eqref{eq_euler_hypergeo_transf} gives that this is equal to 
\[
\frac{\Gamma(k+1)\Gamma(1-s)}{\Gamma(2+k-s)}
(1-x)^{1-s-m} {}_2F_1(2+k-s-m,1-s;2+k-s;x).
\]
Now since $s+m-1 > 0$, the function 
${}_2F_1(2+k-s-m,1-s;2+k-s;z)$
converges uniformly on 
$|z| \le 1$, and so
${}_2F_1(2+k-s-m,1-s;2+k-s;x)$ is bounded for $|x| \le 1$.
Thus, 
the above displayed expression is less than or equal to 
\( C_1(s,m,k) (1-x)^{1-s-m}.\)
\end{proof}
Note that if $2+k > s+m$ and $2+k > s$, then the hypergeometric function 
${}_2F_1(2+k-s-m, 1-s; 2+k-s; x)$ is increasing on $[0,1)$, and 
so the maximum in the bound occurs at $x=1$.  By 
\eqref{eq_hypergeo_value_1},
$C_1(s,m,k)$ becomes
\[
\frac{\Gamma(2+k-s) \Gamma(s+m-1) \Gamma(k+1) \Gamma(1-s)}
{\Gamma(1+k) \Gamma(m) \Gamma(2+k-s)}
=
\frac{\Gamma(s+m-1) \Gamma(1-s)}
{\Gamma(m)}.
\]

\section{Bounds on Integral Means of Bergman Projections}

As discussed above, we make the following definition. 
\begin{definition} Let $f \in L^p(\Disc)$ for $0 < p \le \infty$.  Define
\[
M_p(r,f) = \left\{ \frac{1}{2\pi} \int_0^{2\pi} |f(re^{i\theta})|^p d\theta
\right\}^{1/p}
\]
for $0 < p < \infty$ and 
\[ M_p(r,f) = \esssup_{0 \le \theta < 2\pi} |f(re^{i\theta})| \]
for $p = \infty$. 
\end{definition}
Note that for $f \in L^p(\Disc)$, 
the integral means $M_p(r,f)$ are defined for almost 
every $r$ such that $0 < r < 1$, and in fact the 
function $M_p(\cdot, f)$ is in 
$L^p(r \, dr)$ on $[0,1)$.   For 
$0 < p < \infty$ this follows immediately from Fubini's theorem
(see \cite{Rudin_Big}, Theorem 7.12),  
and for $p = \infty$ it may be proved either directly with the aid 
of Fubini's theorem, or by noting that 
$M_\infty(r,f) = \lim_{p \rightarrow \infty} M_p(r,f)$. 

We say that a function 
$f \in L^p(\partial \mathbb{D})$ is 
in the Sobolev space 
$W^{k,p}(\partial \mathbb{D})$ if 
for every $1 \leq n \leq k$ 
and every function $g \in C^\infty(\partial \mathbb{D})$, 
there is a function $h_n$ such that  
$\int_0^{2\pi} f(e^{i\theta}) \tfrac{d^n}{d\theta^n} g(e^{i\theta}) \, d\theta 
= (-1)^{n}
 \int_0^{2\pi} h_n(e^{i\theta}) g(e^{i\theta}) \, d\theta,$
and furthermore $f$ and each $h_n$ are in $L^p(\partial \mathbb{D}).$ 
Then $h_n$ is unique (see \cite[Chapter 5]{Evans}), so 
we denote $h_n$ by $\tfrac{d^n}{d\theta^n} f$.  
It is well known that if $f \in W^{k,p}$ 
for $1 < p < \infty$ then $f \in C^{k-1}$(see for example 
\cite[Section 5.6]{Evans}).  In fact, since the dimension here is $1$, 
this assertion is not difficult to show directly, and also follows 
for $p=1$.

We next define an auxiliary operator which we will use to help 
bound the Bergman projection.
\begin{definition}
Let $f \in L^1(\partial \Disc).$ Define 
\[
\bp_r^{(n)}(f)(\theta) = \frac{(n+1)!}{2\pi} \int_0^{2\pi} 
        \frac{f(e^{i\phi})r^n e^{-in\phi} }{(1-re^{i(\theta-\phi)})^{2+n}} d\phi.
\]
\end{definition}

We now have the following theorem, which gives a bound on the 
$L^p$ norm of $\bp_r^{(n)}(f)$.  
\begin{theorem}\label{thm:Prn_bound}
Let $1 \le p \le \infty$, and let $k$ be an integer such that $0 \le k \le n$.  
Assume that $f$ is in the Sobolev space $W^{k,p}(\partial \mathbb{D})$. 
Then 
\[
\|\bp_r^{(n)}(f)\|_p \le 
\frac{\Gamma(n+1-k)\Gamma(n+2-k)}{\Gamma((n+2-k)/2)^2}
r^{n-k} (1-r^2)^{k-n-1} \left\| \frac{d^k}{d\theta^k} 
\left[ f(e^{i\theta}) e^{-in\theta}\right]\right\|_p,
\]
where $\|\cdot\|_p$ denotes the $L^p(\partial \Disc)$ norm. 
\end{theorem}
\begin{proof}
First assume that $p < \infty$. 
Performing integration by parts $k$ times gives
\[
\begin{split}
\bp_r^{(n)}(f)(\theta) &= \frac{(n+1)!}{2\pi} \int_0^{2\pi} 
        \frac{f(e^{i\phi}) r^n e^{-in\phi}}{(1-re^{i(\theta-\phi)})^{2+n}} d\phi
\\ &=
r^{n-k} e^{-ik\theta} \frac{(n-k + 1)!}{2\pi i^k} \int_0^{2\pi} 
        \frac{\frac{d^k}{d\theta^k}[f(e^{i\phi}) 
e^{-in\phi}]}{(1-re^{i(\theta-\phi)})^{2+n-k}} d\phi.
\end{split}
\]
This is legitimate since $f$ is in $W^{k,p}$, and thus all its derivatives 
except possibly the $k^{th}$ are continuous.  
We have also used the fact that both 
$f(e^{i\phi})$ and $(1-re^{i(\theta - \phi)})^{-1}$ are periodic 
in $\phi$ with period $2\pi$. 

The above displayed equation, Lemma \ref{lemma:1minusrep} and 
H\"{o}lder's inequality immediately gives the case $p = \infty$. 
If $p < \infty$, 
let $m=n+2-k$, and let 
$g(e^{i\theta}) = {\frac{d^k}{d\phi^k}[f(e^{i\phi}) e^{-in\phi}]}$. 
Note that
\[
(r^{n-k} (n-k+1)!)^{-p} \|\bp_r^{(n)}(f)\|_p^p  \le 
\int_0^{2\pi} \left| \int_0^{2\pi}
  \frac{|g(e^{i\phi})|}{|1-re^{i(\theta - \phi)}|^m} \, 
    \frac{d\phi}{2\pi}
  \right|^p \frac{d\theta }{2\pi}.
\]
But the right hand side of the above inequality equals 
\[
\int_0^{2\pi} \left| \int_0^{2\pi}
  \frac{|g(e^{i\phi})|}{|1-re^{i(\theta - \phi)}|^{m/p}} 
  \frac{1}{|1-re^{i(\theta - \phi)}|^{m/q}}
\, \frac{d\phi}{2\pi}
  \right|^p \frac{d\theta }{2\pi},
\]
where $q$ is the conjugate exponent to $p$. 
By H\"{o}lder's inequality, this is less than or equal to 
\[
\int_0^{2\pi} \int_0^{2\pi}
  \frac{|g(e^{i\phi})|^p}{|1-re^{i(\theta - \phi)}|^m} \, \frac{d\phi}{2\pi}
 \left( \int_0^{2\pi}\frac{1}{|1-re^{i(\theta - \phi)}|^{m}}
\, \frac{d\phi}{2\pi}
  \right)^{p/q} \frac{d\theta}{2\pi}.
\]
And by Lemma \ref{lemma:1minusrep}, this is at most
\[
\left(\frac{\Gamma(m-1)}{\Gamma(m/2)^2}\right)^{p-1}
\int_0^{2\pi} \int_0^{2\pi}
  \frac{|g(e^{i\phi})|^p}{|1-re^{i(\theta - \phi)}|^m} \, 
\frac{d\phi}{2\pi}
(1-r^2)^{(1-m)(p-1)}
  \frac{d\theta}{2\pi},
\]
where we have used the fact that $p/q = p-1$.
Now Tonelli's theorem shows that this equals
\[
\begin{split}
& \phantom{={}} \left(\frac{\Gamma(m-1)}{\Gamma(m/2)^2}\right)^{p-1}
(1-r^2)^{(1-m)(p-1)} 
\int_0^{2\pi} |g(e^{i\phi})|^p 
\int_0^{2\pi} \frac{1}{|1-re^{i(\theta - \phi)}|^m} 
\frac{d\theta}{2\pi} \frac{d\phi}{2\pi} 
\\
&\le 
\left(\frac{\Gamma(m-1)}{\Gamma(m/2)^2}\right)^{p-1}
\frac{\Gamma(m-1)}{\Gamma(m/2)^2}
(1-r^2)^{(1-m)(p-1)}
(1-r^2)^{1-m} \int_0^{2\pi} |g(e^{i\phi})|^p \frac{d\phi}{2\pi}
\\
&=
\left(\frac{\Gamma(m-1)}{\Gamma(m/2)^2}\right)^p
(1-r^2)^{(1-m)p} 
\|g\|_p^p,
\end{split}
\]
where we have again applied Lemma \ref{lemma:1minusrep}.
This proves the result for $p < \infty$.  
(Note that in the case $p = 1$, the above proof still works and really 
only involves Lemma \ref{lemma:1minusrep} and Tonelli's theorem, but 
not H\"{o}lder's inequality.)  

\end{proof}

For $f \in L^1(\mathbb{D})$, recall that the Bergman projection of $f$ is 
defined by 
\[
\bp f(z) = \frac{1}{\pi} 
    \int_{\Disc} \frac{f(w)}{(1-\conj{w}z)^{2}} dA(w),
\]
and thus
\begin{equation}\label{eq:bp_deriv}
\frac{d^n}{dz^n} (\bp f)(z) = \frac{(n+1)!}{\pi} 
\int_{\Disc} \frac{f(w)\conj{w}^n}{(1-\conj{w}z)^{2+n}} dA(w).
\end{equation}
Therefore, if $z=re^{i\theta}$, we have that 
\[
\begin{split}
\frac{d^n}{dz^n}(\bp f)(z) &= \frac{(n+1)!}{\pi}\int_0^1 \rho \int_0^{2\pi} 
           \frac{f(\rho e^{i\phi})\rho^n e^{-in\phi}}{(1-r \rho e^{i(\theta -
                 \phi)})^{2+n}}\, d\phi \, d\rho
\\
&= \frac{(n+1)!}{\pi}
\int_0^1 \rho \int_0^{2\pi} 
            \frac{f_\rho(e^{i\phi}) \rho^n e^{-in\phi}}
                 {(1-r \rho e^{i(\theta -
               \phi)})^{2+n}}\, d\phi \, d\rho,
\\
&=
2
\int_0^1 \rho r^{-n} \bp_{r\rho}^{(n)}f_\rho (e^{i\theta})\, d\rho,
\end{split}
\]
where $f_\rho(e^{i\theta}) = f(\rho e^{i\theta})$.

\begin{theorem}\label{thm:BPn_bound}
Let $1 \le p \le \infty$, and
let $k$ and $n$ be integers such that $0 \le k \le n$.  
Suppose that $f \in L^1(\mathbb{D})$, and that the restriction of $f$ to 
almost every circle of radius less than $1$ 
centered at the origin is in $W^{k,p}$.  
Then the following inequality holds: 
\[
\begin{split}
M_p \left(r, \frac{d^n}{dz^n}(\bp f(z))\right )  &\le 
2 \frac{\Gamma(n+1-k)\Gamma(n+2-k)}{\Gamma((n+2-k)/2)^2} 
\quad \times \\
&\ r^{-k} \int_0^1 \rho^{n+1-k} 
M_p\left( \frac{d^k}{d\theta^k} (e^{-in\theta} f), 
   \rho \right) (1-r^2\rho^2)^{k-n-1}  
d\rho.
\end{split}
\]
\end{theorem}

We make the following remark about this theorem: since 
\begin{equation}\label{eq:simple_deriv_inequality}
\left| \frac{d^k}{d\theta^k} (e^{-in\theta} f) \right|
\le
\sum_{j=0}^k \binom{k}{j} n^{k-j} \left| \frac{d^j}{d\theta^j} f \right|,
\end{equation}
it is not hard to use the above theorem to bound $M_p(r, (\bp f)^{(n)})$ 
strictly in terms of the integral means of the first $k$ derivatives of 
$f$ in the $\theta$ direction. 
\begin{proof}
Again, first assume that $p < \infty$.  We have that  
\[
\begin{split}
M_p \left(r, \frac{d^n}{dz^n}(\bp f(z))\right )
= \left( \frac{1}{2\pi} \int_0^{2\pi} \left| 
2 \int_0^1 \rho r^{-n} \bp_{r\rho}^{(n)}f_\rho (e^{i\theta}) d\rho
 \right|^p d\theta
\right)^{1/p}.
\end{split}
\]
By Minkowski's inequality, this is less than or equal to 
\[
2 \int_0^1 \left( \frac{1}{2\pi} \int_0^{2\pi} 
\rho^p r^{-pn} |\bp_{r\rho}^{(n)}f_\rho (e^{i\theta})|^p
  d\theta \right)^{1/p} d\rho.
\]
By Theorem \ref{thm:Prn_bound}, this is less than or equal to 
\[
2 \frac{\Gamma(n+1-k)\Gamma(n+2-k)}{\Gamma((n+2-k)/2)^2}
r^{-k}  \int_0^1 \rho^{n+1-k} \left\| 
   \frac{d^k}{d\theta^k} (e^{-in\theta} f_\rho) \right\|_p 
    (1-r^2\rho^2)^{k-n-1}  
      d\rho
\]
which equals
\[
\begin{split}
2 \frac{\Gamma(n+1-k)\Gamma(n+2-k)}{\Gamma((n+2-k)/2)^2}
r^{-k}  \int_0^1 \rho^{n+1-k} 
M_p\left( \frac{d^k}{d\theta^k} (e^{-in\theta} f), \rho 
\right) (1-r^2\rho^2)^{k-n-1}  
d\rho.
\\
\end{split}
\]
The proof is slightly easier in the case $p = \infty$, as we 
do not need Minkowski's inequality. 
Alternately, to see that the theorem still holds for $p = \infty$, 
we can take the 
limit in the bound as $p = \infty$, using the monotone convergence 
theorem and the fact that $M_p(r,f)$ increases with $p$. 
\end{proof}

We now discuss Lipschitz and Lebesgue-Lipschitz classes, since they are 
relevant to some corollaries which we are about to prove.
A function $f$ is said to be Lipschitz of order $\alpha$ 
for $0 < \alpha \le 1$ if there is 
some constant $A$ such that $|f(x) - f(y)| \le A|x-y|^\alpha$ for all 
$x$ and $y$ in its domain.  
The class of all such functions is denoted by $\Lambda_\alpha$. 
For a function $f$ defined on the unit circle, we define its 
integral modulus of continuity of order $p$ for $p < \infty$ by 
\[
\omega_p(t,f) = \sup_{0 < h \le t}
                \left[ \frac{1}{2\pi} \int_0^{2\pi} 
                |f(x+h) - f(x)|^p \right]^{1/p}.
\]
If $\omega_p(t,f) = O(t^\alpha)$ for some $\alpha$ such that 
$0 < \alpha \le 1$, we say that $f$ belongs to the 
Lebesgue-Lipschitz class $\Lambda_{\alpha,p}$. 
For $p = \infty$, we define $\Lambda_{\alpha,\infty} = \Lambda_\alpha$. 

We will need
Theorem 5.4 in 
\cite{D_Hp}, which states that an analytic function is in 
$H^p$ for $1 \le p < \infty$ 
and has boundary values in $\Lambda_{\alpha}^p$ if and only if 
the integral means of its derivative satisfy 
$M_p(r,f') = O((1-r)^{-1+\alpha})$.
We will also use Theorem 5.1 from the same reference, which says that 
an analytic function is in 
$H^\infty$ and has boundary values in 
$\Lambda_{\alpha,\infty}$ if and only if 
the integral means of its derivative 
satisfy $M_\infty(r,f') = O((1-r)^{-1+\alpha})$.
(As stated, the theorem has the function being continuous in 
$\overline{\Disc}$ in place of its being in $H^\infty$, but any 
analytic function in $H^\infty$ with Lipschitz (or even continuous) 
boundary values is continuous in $\overline{\Disc}$.)
\begin{theorem}\label{thm:bp-lipschitz}
Let $1 \le p \le \infty$ and $0 < \alpha < 1$ and $n \geq 1$.
Suppose that $f$ is measurable in $\Disc$ and that the restriction of $f$ to 
almost every circle of radius less than $1$ centered at the origin is 
in $W^{n,p}$. 
Suppose that 
$M_p \left( \frac{d^{n}}{d\theta^{n}} f, r \right) = O( (1-r)^{-1+\alpha})$ 
Then $\bp(f)^{(n-1)} \in H^p$, and in fact the boundary values of 
$\bp(f)^{(n-1)}$ are 
in the Lebesgue-Lipschitz space $\Lambda_{\alpha,p}$ .
\end{theorem}
\begin{proof}
Note that the assumptions imply that 
$M_p(r, \frac{d^n}{d\theta^n}(e^{-i\theta}f)) = O((1-r)^{-1+\alpha})$, 
and that $f \in L^1(\Disc)$. 
By the above theorem,
\[ 
\begin{split}
M_p(r, \bp(f)^{(n)}) &\le C \int_0^1 (1-\rho)^{-1+\alpha} (1-r^2 \rho^2)^{-1}
   \rho \, d \rho \\
&\le C \int_0^1 (1-\rho)^{-1+\alpha} (1-r \rho)^{-1} \rho \, d \rho
\end{split}
\]
for $r$ near 
enough to $1$, where $C$ is a constant.
By Lemma \ref{yam1}, the above expression is less than or equal to 
$C (1-r)^{\alpha}$ for $r$ near enough to $1$, where 
$C$ is another constant.      
But this implies that 
$\bp(f)^{(n-1)} \in H^p$ and has boundary values in $\Lambda_{\alpha}^p$.
\end{proof}

We now state a corollary related to our original motivation for 
studying this problem.  If we are given an $f \in A^p$ with 
unit norm, 
where $1 < p < \infty$ and $p$ has conjugate exponent $q$, then 
there is a function $k \in A^q$ (unique up to a positive scalar multiple)
such that $f$ solves the extremal problem of maximizing 
$\Rp \int_{\Disc} g \overline{k} \, d\sigma$ among all functions $g$ 
of unit $A^p$ norm. 
The broad question that first motivated our study was: 
if we know that $f$ has certain regularity, can 
we say anything about regularity properties for $k$?
The next corollary is an example of this. 
\begin{corollary}
Let $2 \le p < \infty$, 
let $p-1 \le s \le \infty$, 
let $q$ be the conjugate exponent to $p$, and let $0 < \alpha < 1$. 
Let $f$ be analytic and suppose that 
$f \in H^s$ with boundary values in 
$\Lambda_{\alpha}^s$, and that 
$\|f\|_{A^p}=1$. 
Let $k$ a function in $A^q$ such that $f$ solves the 
extremal problem of finding a function $g$ of unit $A^p$ norm 
maximizing $\Rp \int_{\Disc} g \overline{k} \, d\sigma$. 
Then $k \in H^{s/(p-1)}$ and the boundary values of $k$ are 
in $\Lambda_{\alpha, s/(p-1)}$.
\end{corollary}
\begin{proof}
By the above mentioned Theorem 5.4 from \cite{D_Hp}, we have that 
$M_s(r,f') = O((1-r)^{-1+\alpha})$. 
Now, if we write $f = u + iv$, we have 
\[
\frac{\partial}{\partial \theta} (|f|^{p-2} f) = 
(p-2)|f|^{p-4}[u u_\theta + v v_\theta]f + |f|^{p-2} f_\theta.
\]
The absolute value of the above expression is bounded by 
\(
(p-1)|f|^{p-2} |f'|,
\)
where we have used the 
fact that $f_{\theta} = iz f'$ and the  
Cauchy-Schwarz inequality applied to 
$\langle u_\theta, v_\theta \rangle$ and $\langle u, v \rangle.$
Thus we have 
\[
M_{s/(p-1)}\left(r,\frac{\partial}{\partial \theta} (|f|^{p-2} f)\right) 
\le (p - 1) \|f\|_{H^s}^{p-2} M_s(r,f') 
\le C (1-r)^{-1 + \alpha},
\]
where in the first inequality we have used 
H\"{o}lder's inequality, and in the second we have used the hypothesis 
about the growth of the integral means of $f'$. 
Also, it is clear that $M_{s/(p-1)}(r, |f|^{p-2} f)$ is bounded. 
By Theorem \ref{thm:bp-lipschitz}, this implies that 
$\bp(|f|^{p-2} f) \in H^{s/(p-1)}$ and that 
$\bp(|f|^{p-2} f)$ has boundary values 
in $\Lambda^{s/(p-1)}_{\alpha}$. 
But since $\bp(|f|^{p-2} f)$ is a constant multiple of $k$, the 
corollary holds. 
\end{proof}

The next corollary is an analogous result for higher regularity. 
\begin{corollary} Let $n \geq 0$ be an integer, and suppose that 
$n + 1 \leq p < \infty$ and that $p-1 \leq s \leq \infty$. 
Also suppose that $f^{(n-1)} \in \Lambda^s_\alpha$, where 
$0 < \alpha < 1$, and that
$\|f\|_{A^p}=1$. 
Let $k$ a function in $A^q$ such that $f$ solves the 
extremal problem of finding a function $g$ of unit $A^p$ norm 
maximizing $\Rp \int_{\Disc} g \overline{k} \, d\sigma$. 
Then $k^{(n-1)} \in H^{s/(p-1)}$ and the boundary values of $k^{(n-1)}$ are 
in $\Lambda_{\alpha, s/(p-1)}$.
\end{corollary}
\begin{proof}
We have that $\partial^n/\partial \theta^n (f^{p/2} \overline{f}^{(p/2)-1})$ 
equals
\[
\sum_{k=0}^n \binom{n}{k} \partial_\theta^k (f^{p/2}) 
\partial_\theta^{n-k} (\overline{f})^{p/2}.
\]
Now each term of $\partial_\theta^k (f^{p/2})$ is of the form 
$C f^{\alpha} g_1 g_2 \cdots g_m$ where $C$ is a constant, 
each $g_m$ is some $\theta$
derivative of $f$ of order at most $n$, and $\alpha + m = p/2$. Also, only 
one $g_j$ can be a $k^{\text{th}}$ derivative. 
Each term of $\partial_\theta^k (\overline{f}^{p/2-1})$ is of the form 
$C \overline{f}^{\alpha} g_1 g_2 \cdots g_m$ where each $g_j$ is some 
derivative of $\overline{f}$ of order at most $n-k$, and $\alpha + m = p/2$. 
Also, only 
one $g_j$ can be a $(n-k)^{\text{th}}$ derivative. 

Thus, each term of  
$\partial^n/\partial \theta^n (f^{p/2} \overline{f}^{(p/2)-1})$ 
is of the form 
$C f^{\alpha} \overline{f}^\beta g_1 g_2 \cdots g_m$, where 
$\alpha + \beta + m = p - 1$ and each $g_j$ is a $\theta$ derivative of 
either $f$ or $\overline{f}$ of order at most $n$.  
Note that because $p \geq n + 1$, we have that 
$\alpha + \beta$ is nonnegative. 
Then $f^\alpha \conj{f}^{\beta} \in L^{s/(\alpha + \beta)}$ and each $g_j$ 
is in $H^s$, unless $g_j$ is an $n^{\text{th}}$ derivative. Thus, all terms 
(except for the exceptional ones with an $n^{\text{th}}$ derivative) 
are in $H^{s/(p-1)}$, by H\"{o}lder's inequality.  
Also, since $M_p(r,\partial^n_\theta f) = O((1-r)^{-1+\alpha})$, 
H\"{o}lder's inequality again shows that 
the terms with an $n^{\text{th}}$ derivative 
have integral means of order $s/(p-1)$ that are 
$O((1-r)^{-1+\alpha})$.  The result now follows as in the above corollary. 
\end{proof}

These corollaries are similar to 
Theorem 4.3 in \cite{tjf2}, which is proved by very different methods.  
That theorem is only proved 
for $p$ an even integer.  It requires us assume that 
$f \in H^s$ for $p-1 < s < \infty$, 
and yields that $k \in H^{s/(p-1)}$.  
Whether Theorem 4.3 from 
\cite{tjf2} holds when $p$ is not an even integer is still an 
open question.

\section{Bounds on Sobolev norms of Bergman Projections}

We now illustrate how our previous results can be used to 
bound certain weighted $L^p$ norms of derivatives of 
Bergman projections by other weighted $L^p$ norms of 
$\theta$ derivatives of the original function. 
We will need the following lemma. 
\begin{lemma}\label{yam2}
Suppose $1<p<\infty$ and that $j,k > -1$ and $m > 0$
and $u < 1$, and that
$u > 1 - mp$.
Set $w = u + (m-1)p$.  
For a measurable function $f$ define 
\[
g(x) = \int_0^1 \frac{|f(y)|}{(1-xy)^m} y^k \, dy,
\]
where we allow $g(x)$ to take on $\infty$ as a value. 
Then 
\[\|g\|_{L^p(x^j(1-x)^{-u}\, dx)} \le 
C_2 \|f\|_{L^p(x^k(1-x)^{-w}\, dx)},\]
where $L^p$ spaces in the bound 
are on the interval $[0,1]$, and where  
\[
\begin{split}
&C_2 = C_2(p,m,k,j,u) = \\
&\inf_{\substack{
   0 < b < m\\ \text{$a$ satisfies all of 
   \eqref{eq:4aconditions}}}}
C_1(aq, (m-b)q, k)^{1/q} C_1(ap+(m-b)p+u-(p/q), bp, j)^{1/p}.
\end{split}
\]
\end{lemma}
\begin{proof}
Let $q$ be the conjugate exponent to $p$. 
Choose $b$ so that $0 < b < m$. 
First note that the above conditions imply that 
\begin{subequations}
\begin{align}
1-m-\frac{u}{p} & < \frac{1}{q}\\
1-\frac{u}{p}-m & < 1 - \frac{u}{p} - (m - b)\\
\frac{1}{q}-(m-b) & < \frac{1}{q} \\
\frac{1}{q} - (m-b) &< 1 - \frac{u}{p} - (m-b).
\end{align}
\end{subequations}
Thus we can find a number $a$ satisfying
\begin{subequations}\label{eq:4aconditions}
\begin{align}
1-m-\frac{u}{p}  < & a \label{eq:4aconditions-agtnob}\\
\frac{1}{q}-(m-b)  < & a \label{eq:4aconditions-agtb}\\
& a  < \frac{1}{q} \label{eq:4aconditions-altnob}\\
& a   < 1 - \frac{u}{p} - m + b.\label{eq:4aconditions-altb}
\end{align}
\end{subequations}

We may assume without loss of generality that $f \ge 0$, since if 
the inequality holds for $|f|$ it holds for $f$. Now 
\[
\begin{split}
&\phantom{={}}\int_0^1 \frac{f(y)}{(1-xy)^m}y^k\, dy 
\\
&= 
\int_0^1 \frac{f(y)(1-y)^{a}}{(1-xy)^{b}} 
              \frac{(1-y)^{-a}}{(1-xy)^{m-b}} y^k dy
\\ &\le 
\left[ \int_0^1 \frac{ |f(y)|^p (1-y)^{ap}}{(1-xy)^{bp}} y^k dy
\right]^{1/p} \times
     \left[ \int_0^1 \frac{(1-y)^{-aq}}{(1-xy)^{(m-b)q}} y^k dy\right]^{1/q},
\end{split}
\]
by H\"{o}lder's inequality. 
But by Lemma \ref{yam1}, the above expression is less than or equal to 
\[
C_{1,1}^{1/q}
 (1-x)^{(1/q)-a-(m-b)} 
   \left[ \int_0^1 \frac{ |f(y)|^p (1-y)^{ap}}{(1-xy)^{bp}} y^k dy
     \right]^{1/p},
\]
where $C_{1,1} = C_1(aq, (m-b)q, k)$. 
This is valid because $aq + (m-b)q > 1$ and $aq < 1$, which follow from 
inequalities \eqref{eq:4aconditions-agtb} and \eqref{eq:4aconditions-altnob}.
So then 
\[
\begin{split}
&\phantom{={}} \|g\|^p_{L^p(x^j (1-x)^{-u}\, dx)} 
\\& = 
\int_0^1 \left| \int_0^1 \frac{f(y)}{(1-xy)^m}y^k\, dy\right|^p 
(1-x)^{-u} x^j \, dx \\&\le 
C_{1,1}^{p/q}
      \int_0^1 (1-x)^{(p/q)-ap-(m-b)p}
        \int_0^1 \frac{|f(y)|^p (1-y)^{ap}}{(1-xy)^{bp}} y^k \, dy \, 
                (1-x)^{-u} x^j \, dx\\
   &= 
C_{1,1}^{p/q}
   \int_0^1 |f(y)|^p (1-y)^{ap} \int_0^1
   \frac{(1-x)^{(p/q)-ap-(m-b)p-u}}{(1-xy)^{bp}} \,x^j \, dx \, y^k \,dy,
\end{split}
\]
by Tonelli's theorem for nonnegative functions.  
Applying the previous lemma again 
we see that this is less than or 
equal to 
\[
\begin{split}
&\phantom{={}}
C_{1,1}^{p/q} C_{1,2}
\int_0^1 |f(y)|^p (1-y)^{ap} (1-y)^{1+(p/q)-ap-(m-b)p-u-bp} y^k \, dy \\
&= C_{1,1}^{p/q} C_{1,2}
 \int_0^1 |f(y)|^p (1-y)^{-w} y^k \, dy  \\
&= C_{1,1}^{p/q} C_{1,2} \|f\|^p_{L^p(x^k(1-x)^{-w}\, dx)}
\end{split}
\]
where $C_{1,2} =  C_1(ap+(m-b)p+u-(p/q), bp, j)$
This works because $u + (m-b)p + ap - \tfrac{p}{q} < 1$ and 
$u + mp + ap - \tfrac{p}{q} > 1$, 
which follow from 
inequalities \eqref{eq:4aconditions-altb} and \eqref{eq:4aconditions-agtnob}, 
respectively. 
\end{proof}
Note that the proof works even if $g$ is equal to $\infty$ for 
some $x$, since 
H{\"o}lder's inequality holds even if the left or right sides are infinite, 
and Tonelli's theorem holds even if some of the integrals involved 
are infinite. 

One important case is when $j=k=m=1$ and $u=0$.  
In this case we can 
choose $a=1/(pq)$ and $b=1/p$, and then we see that 
$
C_2(p,1,1,1,0) \le C_1(1/p, 1, 1)^{1/q} C_1(1/q, 1, 1)^{1/p}.$
(We have tried to find a choice of $a$ and $b$ yielding a better 
bound on $C_2$, but were not able).
But by the remarks after Lemma \ref{yam1}, this is equal to 
\[
\left[ \frac{\Gamma(1/p) \Gamma(1/q)}{\Gamma(1)} \right]^{1/q}
\left[ \frac{\Gamma(1/q) \Gamma(1/p)}{\Gamma(1)} \right]^{1/p}
= \Gamma\left(\frac{1}{p}\right) \Gamma\left(1 - \frac{1}{p} \right) 
= \frac{\pi}{\sin(\pi / p)}
\]
by the reflection formula for the $\Gamma$ function.

It is interesting to note that the bounds in the following theorem do 
not depend on $p$. 
\begin{theorem}\label{yam3}
Let $1 \le p \le \infty$ and $1<s<\infty.$  
Suppose that $0 \le k \le n$, where $n$ and $k$ are integers.  Also 
suppose that $j -k > -1$ and $1-(n+1-k)s < u < 1$ and set $w = u + (n-k)s$. 
Also suppose that
the restriction of $f$ to 
almost every circle of radius less than $1$ centered at the origin 
is in $W^{k,p}$, and 
that $f$ is in $L^1(\Disc)$. 
Then
\[
\begin{split}
&\phantom{={}}
\left\{ \int_0^1 [M_p(r, (\bp f)^{(n)})]^s \, (1-r)^{-u} r^j \, 
                     dr \right\}^{1/s} 
\\ &\le 
C_3(s,n-k,j-k,u)
\left\{
 \int_0^{1} \left[M_p \left(\frac{d^k}{d\theta^k}(e^{-in\theta}f),r 
             \right)\right]^s 
     (1-r)^{-w} r^{n-k+1} \, dr
\right\}^{1/s}
\end{split}
\]
where 
\[
C_3(s,n-k,j-k,u) = \frac{\Gamma(n+1-k)\Gamma(n+2-k)}{\Gamma((n+2-k)/2)^2}
C_2(s,n-k+1,n-k+1,j-k,u). 
\]
\end{theorem}
\begin{proof}
Define 
\[
 g(r) = \int_0^1 \rho^{n+1-k} M_p\left( \frac{d^k}{d\theta^k} 
         (e^{-in\theta}f),\rho \right)
              (1-\rho r)^{k-n-1} d\rho.
\]
Then by Theorem \ref{thm:BPn_bound} and the fact that 
$(1-\rho^2 r^2)^{k-n-1} \le (1 - \rho r)^{k-n-1}$ we have 
\[
M_p(r, (\bp f)^{(n)}) \le 
C r^{-k} g(r),
\]
where $C = 2 \tfrac{\Gamma(n+1-k)\Gamma(n+2-k)}{\Gamma((n+2-k)/2)^2}.$
But by Lemma \ref{yam2},
\[
\begin{split}
&\phantom{={}} \left( \int_0^1 |g(r)|^s \, (1-r)^{-u} \, r^{j-k} dr \right)^{1/s}  
\\ &\le C_2(s,n-k+1,n-k+1,j-k,u) \quad \times \\ 
&\qquad \qquad \left[
 \int_0^{1} \left[
   M_p \left(\frac{d^k}{d\theta^k}(e^{-in\theta}f),r \right) 
       \right]^s (1-r)^{-w} r^{n-k+1} \, dr
\right]^{1/s}.
\end{split}
\]
\end{proof}
By using Equation \eqref{eq:simple_deriv_inequality}, it is not 
hard to modify the bound in the theorem so that it only involves the 
integral means of the first $k$ derivatives of $f$ in the 
$\theta$ direction. 

Note that if we take $1 < s < \infty$ and $n=k$ and $j=1+n$ and 
$u=0$, then by the remarks after Lemma \ref{yam2} we see that 
$C_3(s,0,1,0) \le 2\pi/\sin(\pi s)$. 
If we also take $p=s$ and note that $r \, dr \, d\theta = dA$, 
we have the following corollary.
\begin{corollary}
For $1 < p < \infty$ and $n \ge 0$, if $f \in L^1(\Disc)$ and 
$f$ is in $W^{n,p}$ when restricted to almost every circle of radius less than 
$1$ centered at the origin, we have 
\[ 
\| \bp(f)^{(n)} \|_{L^p(r^n dA)} \le 
2\frac{\pi}{\sin(\pi/p)} \left\| \frac{d^n}{d\theta^n} 
(e^{-in\theta}f) \right\|_{L^p(dA)}.
\]
\end{corollary}

Here is another corollary, which follows from taking 
$p=s$, replacing $n$ with $n + k$ where $n, k \ge 0$, 
and letting  
$j=1 + k$,
and $u=0$.  
\begin{corollary}
For $1 < p < \infty$ and integers $n,k \ge 0$, if 
$f \in L^1(\Disc)$ and 
$f$ is in $W^{k,p}$ when restricted to almost every circle of radius less than 
$1$ centered at the origin we have 
\[ 
\| \bp(f)^{(n + k)} \|_{L^p(r^k dA)} \le 
C_3(p,n,1,0) \left\| \frac{d^k}{d\theta^k}(e^{-i(n+k)\theta}f) 
       \right\|_{L^p(r^n (1-r)^{-np} \, dA)}.
\]
\end{corollary}
If $k = 0$, the right hand side above simplifies to
\[
C_3(p,n,1,0) \| f \|_{L^p(r^n (1-r)^{-np} \, dA)}.
\]
Now, if we take $b = 1/p$ and $a= -n + 1/(pq)$ in the definition of 
$C_2$, we see that 
\[
\begin{split}
&\phantom{={}} C_3(p,n,1,0) \\
& \le 
2 \frac{\Gamma(n+1)\Gamma(n+2)}{\Gamma(1 + (n/2))} \ \times \\
&\phantom{={}} C_1(aq, (n+1-b)q, n+1)^{1/q} 
C_1(ap + (n+1-b)p - (p-1), bp, 1)^{1/p} \\ &=
2 \frac{\Gamma(n+1) \Gamma(n+2) \Gamma(1/p)^{1/q} \Gamma(nq+(1/q))^{1/q} 
\Gamma(1/q)^{1/p} \Gamma(1/p)^{1/p}}
{\Gamma(1 + (n/2))\Gamma(nq+1)^{1/q} \Gamma(1)^{1/p}}
\\ &=
2 \frac{ \Gamma(n+1) \Gamma(n+2)
\Gamma(1/p)\Gamma(1/q)^{1/p} \Gamma(nq + (1/q))^{1/q}}
{\Gamma(1 + (n/2)) \Gamma(nq+1)^{1/q}}.
\end{split}
\]

\section{A Counterexample}
We now give an example of a function $f$ such that $M_2(r,f)$ is 
bounded but $\bp f$ is not in $H^2$.  
In fact, the function in our example can be chosen so that 
$f \in C^{\infty}(\mathbb{D})$ and so that 
$M_\infty (r, f) \rightarrow 0$ as $r \rightarrow 1^{-}$. 
This function is called a ``counterexample'' since it is a counterexample 
to the conjecture that $\bp(f) \in H^2$ if $M_2(r,f)$ is bounded. 
Note that for the special case where $f$ is of the form 
$|g|^{p-2}g$ where $g$ is an analytic function, then 
$\bp(f) \in H^2$ if $M_2(r,f)$ is bounded (see the introduction).  

We first derive some general formulas for the Bergman projection of a 
function. 
Suppose that $f \in L^2(\Disc)$. 
Note that for almost every $r$ in $[0,1]$, $f$ restricted to the circle 
of radius $r$ has a Fourier series
since it is in $L^2([0,2\pi))$ for almost every $r$. 
Thus we can write
\[
f(re^{i\theta}) = \sum_{n=-\infty}^\infty a_n(r) e^{in\theta},
\]
where for a.e.~$r$ convergence holds in $L^2(0, 2\pi)$. 
Here 
\[
a_n(r) = \frac{1}{2\pi} \int_0^{2\pi} f(re^{i\theta}) e^{-in\theta}\,  d\theta.
\]
Note that the functions $a_n(r)$ are measurable by Fubini's theorem. 
Also, by Fubini's theorem
\[
\int_{\Disc} |f(z)|^2 \, \frac{dA(z)}{2\pi} = 
\frac{1}{\pi}
\int_0^1 \int_{0}^{2\pi} |f(r e^{i\theta})|^2 \, d\theta \, r\, dr.
\]
But since, for almost every fixed $r$, the Fourier series in $\theta$ 
of $f(re^{i\theta})$ converges in $L^2(0, 2\pi)$, we have 
\begin{equation}\label{eq:disc_fourier_norm}
\int_{\Disc} |f(z)|^2 \, \frac{dA(z)}{2\pi} = 
\frac{1}{\pi}
\int_0^1 \sum_{n=0}^\infty |a_n(r)|^2 \, r \, dr.
\end{equation}

We now prove the following lemma relating what we have said to 
calculating the Bergman projection of $f$. 
\begin{lemma}
Suppose that $f \in L^2(\Disc)$.  Then we can write
\[
f(re^{i\theta}) = \sum_{n=-\infty}^\infty a_n(r) e^{in\theta}
\]
for a.e.~$r$, 
where for a.e.~$r$, convergence holds in $L^2(0, 2\pi)$. 
Also, the Bergman projection of $f$ is given by
\[
(\bp f)(z) = \frac{1}{\pi} \sum_{n=0}^\infty 
\left[ \int_0^1 (n+1)a_n(r) r^{n+1} dr \right]
z^n.
\]
\end{lemma}
\begin{proof}
Let $z=r e^{i\theta}$ and $w=\rho e^{i\phi}$. Note that 
\[
\begin{split}
(\bp f)(z) &= \int_{\Disc} \sum_{n=0}^\infty (n+1) z^n \conj{w}^n f(w) 
\, d\sigma(w)
\\ &=
\frac{1}{\pi}
\int_0^1 \int_0^{2\pi} \sum_{n=0}^\infty \sum_{m=-\infty}^\infty 
    (n+1) a_m(\rho) e^{im\phi} r^{n} \rho^{n+1} e^{in\theta}e^{-in\phi}
\, d\phi \, d\rho
\\
\end{split}
\]
by Fubini's theorem. For fixed $z$ and $\rho$, the sum 
$ \displaystyle \sum_{n=0}^\infty (n+1) z^n \conj{w}^n$
converges uniformly on $[0,2\pi]$, and thus for 
fixed $z$ and 
almost every fixed $\rho$, the sum 
$\displaystyle 
\sum_{n=0}^\infty (n+1) z^n \conj{w}^n f(w)$
converges in $L^2([0,2\pi])$. 
Also, for almost every fixed $\rho$ 
the sum  
$ \displaystyle \sum_{m=-\infty}^\infty a_m(\rho) e^{im\phi}$
converges in $L^2(0,2\pi)$.  
Thus, we can 
move the integral over $\phi$ inside the two summations to see that  
\[
(\bp f)(z) 
= \frac{1}{\pi}
\int_0^1 \sum_{n=0}^\infty (n+1) a_n(\rho) \rho^{n+1} r^n e^{in\theta} \, d\rho.
\]
Now, we wish to apply the dominated convergence theorem to move the sum 
outside the integral.  To see that we can do this, note that 
for each $\rho$ and each $N \ge 0$ we
have by the Cauchy-Schwarz inequality that 
\[
\left| \sum_{n=0}^N (n+1) a_n(\rho) \rho^{n} r^n e^{in\theta} \right|
\le
\left(\sum_{n=0}^\infty |a_n(\rho)|^2 \right)^{1/2}
\left(\sum_{n=0}^\infty (n+1)^2 r^{2n} \rho^{2n} \right)^{1/2}. 
\]
But the second sum can be bounded by $(1+r^2)/(1-r^2)^3$ 
independently of $\rho$  
and the first sum is 
integrable in with respect to the measure $\rho \, d \rho$ by 
equation \eqref{eq:disc_fourier_norm}.
Thus we may apply the dominated convergence theorem to see 
that 
\[
(\bp f)(z) 
= \frac{1}{\pi}
 \sum_{n=0}^\infty 
\left[ \int_0^1 (n+1) a_n(\rho) \rho^{n+1} \, d\rho \right]
z^n.
\]
\end{proof}
Of course, this theorem shows that the formula for the Bergman 
projection is valid if $f \in L^p$ for $p > 2$, 
since $L^p$ is then a subset of $L^2$.  The formula also holds for 
any $f \in L^p$ for $p > 1$.  This can be shown by using 
the fact that the Fourier series of an $L^p$ function converges to 
that function in $L^p$ for $1 < p < \infty$, by using 
H\"{o}lder's inequality instead of the Cauchy-Schwarz inequality, and 
by using Fubini's theorem and the Hausdorff-Young inequality to show that 
$\sum_{n=0}^\infty |a_n(\rho)|^q$ is integrable with respect to 
$\rho \, d \rho$. However, we will really only need the 
formula to hold for bounded functions, since the 
functions to which we need to apply the theorem will all be bounded. 

To construct a function $f$ such that $M_\infty(r,f) \rightarrow 0$ as 
$r \rightarrow 1^{-}$ and $\bp f \not\in H^2$, 
we will use the following lemma.  Note that the constant $1/4$ in 
the lemma is  
not sharp and could be replaced any number strictly between $0$ and $1$. 
\begin{lemma}
There is an increasing sequence 
$0=b_0, b_1, b_2, \ldots \rightarrow 1$ and an increasing sequence 
of non-negative integers $m_1, m_2, \ldots$ such that 
\[
\int_{b_{n-1}}^{b_n} (m_n + 1) r^{m_n + 1} \, dr \ge \frac{1}{4}.
\]
\end{lemma}
\begin{proof}
We prove this by induction.  Let $b_0=0, b_1=1/\sqrt{2}$, and  $m_1 = 0$.  
Then we have 
\[
\int_{b_0}^{b_1} (m_1 + 1)r^{m_1+1} \, dr = 
\int_0^{1/\sqrt{2}} r \, dr = \frac{1}{4}.
\]

Now, suppose we have found an increasing sequence of constants 
$b_0, \ldots, b_n < 1$  and an increasing sequence of non-negative integers
$m_1, \ldots m_n$ satisfying the above condition. 
Note that for each $k \ge 0$, 
\[
\int_{b_n}^1 (k+1)r^{k+1} \, dr = \frac{k+1}{k+2}(1-b_n^{k+1}).
\]
Now, as $k \rightarrow \infty$, this approaches $1$, so there is some $k$ 
such that 
\[
\int_{b_n}^1 (k+1)r^{k+1} \, dr = \frac{k+1}{k+2}(1-b_n^{k+1}) \ge \frac{1}{2}.
\]
We choose $m_{n+1}$ to be the smallest such $k$.  
Now, the above inequality 
implies that there is some constant $b$ such that 
$b_n < b < 1$ and  
\[
\int_{b_n}^{b} (k+1)r^{k+1} \, dr = \frac{1}{4}.
\]
We then choose $b_{n+1} = b$. 
\end{proof}

We now have the following theorem in which we construct bounded functions 
whose Bergman projections are not in $H^2$. 
\begin{theorem}
Let the $b_n$ be defined as in the previous lemma.  
Let $\{c_n\}_{n=1}^\infty$ be a bounded sequence such that 
$\sum_{n} |c_n|^2 = \infty$. Define
\[
a_j(r) = c_j \chi_{[b_{j-1}, b_j)}(r) 
\]
and
\[
f(re^{i\theta}) = \sum_{j=1}^\infty a_j(r) e^{im_j \theta}.
\]
Then $f$ is bounded but $\bp f$ is not in $H^2$. 
\end{theorem}
\begin{proof}
  For each $r$ there is exactly one $j$ such that 
$r \in [b_{j-1}, b_j)$, which implies that
$M_\infty(r,f) = c_j$, where $j$ is the number such that 
$r \in [b_{j-1}, b_j)$. Note that this implies that 
$f$ is bounded. Thus we have that 
\[
(\bp f)(z) = 
\frac{1}{\pi}
\sum_{n=1}^\infty c_n \left[ \int_{b_{n-1}}^{b_n} (m_n + 1) r^{m_n + 1} \, dr
\right] z^{m_n}.
\]
But this means that the $m_n{}^{\text{th}}$ term in the Taylor 
series of $\bp f$ is at least $c_n/4$ in absolute value, so that the 
Taylor coefficients of $\bp f$ are not square summable. 
\end{proof}

In the theorem, if we choose the sequence $\{c_n\}$ so that 
it approaches $0$ (but is not square summable), then the function 
$f$ defined in the statement of the theorem will approach $0$ uniformly 
as $z$ approaches the boundary of the disc, but its Bergman projection 
will not be in $H^2$.  Thus, we have the following corollary.
\begin{corollary}
There is a bounded function $f(z)$ in $\Disc$ that approaches $0$ uniformly 
as $|z| \rightarrow 1$, such that $\bp f \not \in H^2$. 
\end{corollary}
We note in passing that if we define 
$a_j(r) = c_j \phi_j(r)$, where $\phi_j$ is a $C^\infty$ 
bump function with support in 
$(b_{j-1}, b_j)$ that is equal to $1$ on a sufficiently large part of 
$(b_{j-1}, b_j)$
then we can even construct $f$ so that it is in $C^\infty$ and 
approaches $0$ uniformly as $|z| \rightarrow 1$.

\nocite{Olver:2010:NHMF}
\bibliography{extremal}

\providecommand{\bysame}{\leavevmode\hbox to3em{\hrulefill}\thinspace}
\providecommand{\MR}{\relax\ifhmode\unskip\space\fi MR }
\providecommand{\MRhref}[2]{%
  \href{http://www.ams.org/mathscinet-getitem?mr=#1}{#2}
}
\providecommand{\href}[2]{#2}
\begin{thebibliography}{10}

\bibitem{NIST:DLMF}
\emph{{NIST Digital Library of Mathematical Functions}}, http://dlmf.nist.gov/,
  Release 1.0.6 of 2013-05-06, Online companion to \cite{Olver:2010:NHMF}.

\bibitem{Dostanic_Berezin_Norm}
Milutin Dostani{\'c}, \emph{Norm of {B}erezin transform on {Lp} space}, J.
  Anal. Math. \textbf{104} (2008), 13--23. \MR{2403427 (2009e:30015)}

\bibitem{Dostanic_BP_Norm}
Milutin~R. Dostani{\'c}, \emph{Two sided norm estimate of the {B}ergman
  projection on {$L^p$} spaces}, Czechoslovak Math. J. \textbf{58(133)} (2008),
  no.~2, 569--575. \MR{2411110 (2009f:32003)}

\bibitem{D_Hp}
Peter Duren, \emph{Theory of {$H\sp{p}$} spaces}, Pure and Applied Mathematics,
  Vol. 38, Academic Press, New York, 1970. \MR{0268655 (42 \#3552)}

\bibitem{D_Ap}
Peter Duren and Alexander Schuster, \emph{Bergman spaces}, Mathematical Surveys
  and Monographs, vol. 100, American Mathematical Society, Providence, RI,
  2004. \MR{2033762 (2005c:30053)}

\bibitem{Evans}
Lawrence~C. Evans, \emph{Partial differential equations}, Graduate Studies in
  Mathematics, vol.~19, American Mathematical Society, Providence, RI, 1998.
  \MR{1625845}

\bibitem{tjf:lipschitzbergman}
Timothy Ferguson, \emph{Bergman-h\"{o}lder functions and extremal problems}, In
  preperation.

\bibitem{tjf1}
Timothy Ferguson, \emph{Continuity of extremal elements in uniformly convex
  spaces}, Proc. Amer. Math. Soc. \textbf{137} (2009), no.~8, 2645--2653.

\bibitem{tjf2}
Timothy Ferguson, \emph{Extremal problems in {B}ergman spaces and an extension
  of {R}yabykh's theorem}, Illinois J. Math. \textbf{55} (2011), no.~2,
  555--573 (2012). \MR{3020696}

\bibitem{tjf3}
\bysame, \emph{Solution of extremal problems in {B}ergman spaces using the
  {B}ergman projection}, Comput. Methods Funct. Theory \textbf{14} (2014),
  no.~1, 35--61. \MR{3194312}

\bibitem{Zhu_Ap}
H{\aa}kan Hedenmalm, Boris Korenblum, and Kehe Zhu, \emph{Theory of {B}ergman
  spaces}, Graduate Texts in Mathematics, vol. 199, Springer-Verlag, New York,
  2000. \MR{1758653 (2001c:46043)}

\bibitem{MR2965249}
A.-K. Herbig and J.~D. McNeal, \emph{A smoothing property of the {B}ergman
  projection}, Math. Ann. \textbf{354} (2012), no.~2, 427--449. \MR{2965249}

\bibitem{MR3130312}
A.-K. Herbig, J.~D. McNeal, and E.~J. Straube, \emph{Duality of holomorphic
  function spaces and smoothing properties of the {B}ergman projection}, Trans.
  Amer. Math. Soc. \textbf{366} (2014), no.~2, 647--665. \MR{3130312}

\bibitem{Olver:2010:NHMF}
F.~W.~J. Olver, D.~W. Lozier, R.~F. Boisvert, and C.~W. Clark (eds.),
  \emph{{NIST Handbook of Mathematical Functions}}, Cambridge University Press,
  New York, NY, 2010, Print companion to \cite{NIST:DLMF}.

\bibitem{Rudin_Big}
Walter Rudin, \emph{Real and complex analysis}, McGraw-Hill Book Co., New
  York-Toronto, Ont.-London, 1966. \MR{0210528 (35 \#1420)}

\bibitem{Ryabykh}
V.~G. Ryabykh, \emph{Extremal problems for summable analytic functions},
  Sibirsk. Mat. Zh. \textbf{27} (1986), no.~3, 212--217, 226 ((in Russian)).
  \MR{853902 (87j:30058)}

\bibitem{Dragan_Isoperimetric}
Dragan Vukoti{\'c}, \emph{The isoperimetric inequality and a theorem of {H}ardy
  and {L}ittlewood}, Amer. Math. Monthly \textbf{110} (2003), no.~6, 532--536.
  \MR{1984405}

\end{thebibliography}

\end{document}